\documentclass[12pt,a4paper,leqno]{amsart}

\usepackage{amsfonts,amssymb,mathrsfs,bm,enumerate,color,graphicx}
\usepackage[pdftex]{hyperref}
\usepackage[numbers]{natbib}

\numberwithin{equation}{section}
\theoremstyle{plain}
 \newtheorem{thm}{Theorem}[section]
 \newtheorem{lem}[thm]{Lemma}
 \newtheorem{cor}[thm]{Corollary}
 \newtheorem{prop}[thm]{Proposition}
 
\theoremstyle{definition}
 
 \newtheorem{ex}[thm]{Example}
 \newtheorem{rem}[thm]{Remark}

\newcommand{\al}{\alpha}

\newcommand{\gm}{\gamma}

\newcommand{\dl}{\delta}

\newcommand{\ep}{\varepsilon}

\newcommand{\ld}{\lambda}

\newcommand{\ta}{\tau}

\newcommand{\eqd}{\overset{\mathrm d}{=}}
\newcommand{\sek}{\int_0^{\infty}}
\newcommand{\xtm}{X_t^{(\mu)}}
\newcommand{\wh}{\widehat}
\newcommand{\wt}{\widetilde}
\newcommand{\la}{\langle}
\newcommand{\ra}{\rangle}
\newcommand{\B}{\mathcal{B}}

\newcommand{\R}{\mathbb{R}}

\newcommand{\rd}{{\mathbb R^d}}
\newcommand{\law}{\mathcal L}

\newcommand{\n}{\noindent}

\def\1{\scalebox{0.94}{$1$}\hspace{-0.36em}1}

\allowdisplaybreaks[4]
\begin{document}
\setlength{\baselineskip}{18pt}
\setlength{\parindent}{1.8pc}

\title{Stochastic integral and series representations for strictly stable distributions}

\author[M.~Maejima]{Makoto Maejima$^{1,2}$}
\author[J.~Rosi\'nski]{Jan Rosi\'nski$^{3,4}$}
\author[Y.~Ueda]{Yohei Ueda$^{5,6}$}

\date{\today}

\begin{abstract}
In this paper we find and develop a stochastic integral representation for the class of strictly stable distributions. We establish an explicit relationship between stochastic integral and shot-noise  series representations of strictly stable distributions, which shows that the class of distributions representable by stochastic integral is larger than the class  representable by a shot-noise  series. This inclusion is proper when the stability index $\al$ is greater than 1.  We also give an explicit description of distributions possessing both representations. 
\end{abstract}

\maketitle

\footnotetext[1]{Department of Mathematics, Keio University, 3-14-1, Hiyoshi, 
Kohoku-ku, Yokohama 223-8522, Japan. \email{maejima@math.keio.ac.jp}}
\footnotetext[2]{Maejima's research was partially supported by JSPS Grant-in-Aid for Science Research 22340021.}
\footnotetext[3]{Department of Mathematics, 227 Ayres Hall, University of Tennessee, Knoxville, TN, USA. \email{rosinski@math.utk.edu}}
\footnotetext[4]{Rosi\'nski's research was partially supported by the Simons Foundation grant 281440.}
\footnotetext[5]{Department of Mathematics, Keio University, 3-14-1, Hiyoshi, 
Kohoku-ku, Yokohama 223-8522, Japan. \email{ueda@math.keio.ac.jp}}
\footnotetext[6]{Corresponding author.}

\section{Introduction}

Throughout the paper, $I(\rd)$ denotes the class of all infinitely divisible distributions on $\rd$,
$\law(X)$ stands for the distribution of a random variable $X$, and the point-mass 
distribution at $a\in\rd$ is denoted by $\dl_a$. Given $\mu\in I(\rd)$, $\{X_t^{(\mu)}\}$ will denote a L\'evy process such that $\law(X_1)=\mu$.
For a fixed nonrandom function $f$, consider the stochastic integral mapping $\Phi_f\colon \mathfrak D(\Phi_f) \to I(\rd)$ given by
\begin{equation}\label{si}
\Phi_f(\mu)=\law\left(\int_0^\infty f(t)dX_t^{(\mu)}\right),
\end{equation}
where the domain $\mathfrak D(\Phi_f)$ consists of all $\mu \in I(\rd)$ for which the stochastic integral in (\ref{si}) is definable (see \cite{Sato2007}).

Stochastic integral mappings give probabilistic representations for many useful classes of distributions contained in $I(\rd)$. Namely, such classes can be represented as ranges of the mappings $\Phi_f$ for some specific functions $f$.  
Examples of classes of distributions having stochastic integral representations include the class of selfdecomposable distributions \cite{JurekVervaat1983, SatoYamazato1984, Wolfe1982},
the Jurek class \cite{Jurek1985},
the Goldie-Steutel-Bondesson class and the Thorin class \cite{Barndorff-NielsenMaejimaSato2006},
the class of type $G$ distributions \cite{AoyamaMaejima2007},
and many other classes \cite{AoyamaLindnerMaejima2009, AoyamaMaejimaRosinski2008, MaejimaUeda2009, Sato2010}. 
Characterizations of the domains and ranges of the corresponding stochastic integral mappings allow to determine the extent of such representations. They  are also necessary for the study of 
iterations of stochastic integral mappings and their limits \cite{MaejimaSato2009, MaejimaUeda2009, 
MaejimaUeda2009e, MaejimaUeda2009c, Rocha-ArteagaSato2003, 
Sato2010, Sato2011, Sato2012, SatoUeda2012}, compositions of several mappings 
\cite{Barndorff-NielsenMaejimaSato2006, Sato2006, MaejimaUeda2009e, Sato2010}, and related considerations \cite{AoyamaMaejimaUeda2011, IchifujiMaejimaUeda2010}.

Surprisingly for the present authors, the form of a stochastic integral representation of the class of stable distributions has been unknown. On the other hand, a shot noise-type series representation of such distributions has been known and used by many authors for a long time (see, e.g., \cite{S-T_book1994}).  In this paper, we give a stochastic integral representation of strictly stable distributions on $\rd$, characterize the domain and range of the corresponding stochastic integral mapping, and establish an explicit relationship between stochastic integral and series representations. Our main results are Theorems \ref{main}, \ref{series}, and \ref{Poisson}. 

The stochastic integral representation of strictly stable distributions is given by the mapping $\Xi_{\al}(\mu) := \law\big(\sek t^{-1/\al}d\xtm \big)$.
In Theorem \ref{main} we show that the range of $\Xi_{\al}$ coincides with the class of strictly $\al$-stable distributions  when $\al \ne 1$.  The case $\alpha=1$ is more delicate since the range of $\Xi_{1}$ is smaller than the class of strictly $1$-stable distributions. We describe precisely distributions representable by $\Xi_{1}$, from which it follows that every strictly $1$-stable distribution belongs to the range of $\Xi_{1}$ after certain deterministic translation. 

The relationship between stochastic integral and series representations of strictly stable distributions is captured when one studies the restriction of $\Xi_{\al}$ to the subdomain consisting of compound Poisson 
distributions. For such distributions $\mu$, $X^{(\mu)}_t = \sum_{j\colon\tau_j \le t} V_j$, 
where $\tau_j$ is the $j$th arrival time of a Poisson process and $\{V_j\}$ is a sequence of 
i.i.d.\ random variables in $\rd$, independent of $\{\tau_j\}$. 
In this case, at least formally, we can write 
\begin{equation} \label{LP}
   \sek t^{-1/\al}d\xtm = \sum_{j=1}^{\infty}\tau_j^{-1/\al}V_j.
\end{equation}
The right-hand side is a well-known series representation of stable distributions; 
see, e.g., Corollary 4.10 of \citet{Rosinski1990}, Proposition 1.4.1 of  \cite{S-T_book1994}. 
In Theorem~\ref{series} we investigate \eqref{LP} and characterize the class of distributions 
representable by either side of this equation.  
It follows that the series representation (without centering) is a special case of our stochastic integral representation. 

It is clear that $\Xi_{\al}$ is not injective on its whole domain.  
In Theorem~\ref{Poisson} we consider  $\Xi_{\al}$ restricted to subdomains of
some infinitely divisible distributions with finite L\'evy measures
supported by the unit sphere. On such subdomains, $\Xi_{\al}$ is injective and we characterize the corresponding ranges of  $\Xi_{\al}$.

One special property of $\Xi_{\al}$ is mentioned at the end of the paper in Remark \ref{iteration}. It says that the limit of the ranges of iterations of our mapping consists of only one distribution $\dl_0$.  This property makes $\Xi_{\al}$ quite different from  other  stochastic integral mappings considered in the past.

Finally, we remark that, as far as representations are concerned, 
the restriction to strictly stable distributions is not essential when  
$\al\ne 1$, because any $\al$-stable distribution is strictly stable  up to a convolution 
with a $\dl$-distribution. (Theorem 14.7 in \citet{Sato's_book1999}. 
See also the end of the section 2.)

\section{Preliminaries}

The characteristic function $\wh\mu(z)$, $ z\in \rd$, of $\mu\in I(\rd)$ is given by 
the L\'evy-Khintchine triplet $(A, \nu, \gm)$ as follows
$$
\wh\mu(z) = \exp\left \{ -\frac12 \la z, Az\ra + \int_{\rd} \left ( e^{i\la z,x\ra}
-1 - i\la z,x\ra\1_{\{|x|\leq 1\}}(x)\right ) \nu (dx) + i\la \gm ,z\ra\right \},
$$
where $A$ is a $d\times d$ symmetric nonnegative-definite matrix, $\nu$ is a measure 
(called the L\'evy measure of $\mu$) on
$\rd$ satisfying $\nu(\{0\})=0$ and $\int_{\rd} (|x|^2\wedge 1)\nu(dx)<\infty$ and $\gm\in\rd$.
We will write  $\mu= \mu_{(A,\nu,\gm)}$ to denote an infinitely divisible distribution 
with the L\'evy-Khintchine triplet $(A, \nu, \gm)$.
If $\mu= \mu_{(A,\nu,\gm)}$ satisfies $\int_{|x|> 1}|x|\nu(dx)<\infty$, then
there exists the mean $\gm^1\in \rd$ of $\mu$ such that
$$
\wh\mu(z) = \exp\left \{ -\frac12 \la z, Az\ra + \int_{\rd} \left ( e^{i\la z,x\ra}
-1 -i\la z, x\ra \right ) \nu (dx) + i\la \gm^1 ,z\ra\right \}
$$
and 
$$
\gm^1 = \gm + \int_{|x|>1}x\nu(dx).
$$
In this case, we will write $\mu =\mu _{(A, \nu, \gm^1)_1}$.
If $\mu= \mu_{(A,\nu,\gm)}$ satisfies $\int_{|x|\le 1}|x|\nu(dx)<\infty$, then
there exists $\gm^0\in \rd$ (called the drift of $\mu$) such that
$$
\wh\mu(z) = \exp\left \{ -\frac12 \la z, Az\ra + \int_{\rd} \left ( e^{i\la z,x\ra}
-1 \right ) \nu (dx) + i\la \gm^0 ,z\ra\right \}
$$
and 
$$
\gm^0 = \gm - \int_{|x|\leq 1}x\nu(dx).
$$
We will write $\mu =\mu _{(A, \nu, \gm^0)_0}$ in this case.

\begin{cor}\label{0=1}
Let $\mu =\mu_{(A,\nu,\gamma)}\in I(\rd)$. Then $\mu_{(A, \nu, 0)_0}=\mu_{(A, \nu, 0)_1}$ 
if and only if  $\int_{\rd} |x| \, \nu(dx) < \infty$, $\int_{\rd} x \, \nu(dx) = 0$ 
and $\gamma=  \int_{|x|\le 1} x\nu(dx)$.
\end{cor}

We also use the following polar decomposition of a L\'evy measure $\nu$.
If $\nu\neq 0$, then
there exist a measure $\lambda$ on $S=\{x\in\rd\colon |x|=1\}$ with $0<\lambda(S)<\infty$ and
a family $\{\nu_{\xi}, \xi\in S\}$ of measures on $(0,\infty)$ such that
$\nu_{\xi}(B)$ is measurable in $\xi$ for each $B\in\B((0,\infty))$,
$0<\nu_{\xi}((0,\infty))\le\infty$ for each $\xi\in S$ and
\begin{equation}\label{polar}
   \nu(B)=\int_S \lambda(d\xi)\int_0^{\infty} \1_B(r\xi)\nu_{\xi}(dr),\quad
B\in \B (\rd \setminus \{ 0\}).
\end{equation}
Measure $\lambda$ is  called  the spherical component of $\nu$ and $\nu_\xi$ its radial component. 
If $\nu(\rd)< \infty$, then we may and do assume that $\nu_\xi$ are probability measures. 
Indeed, consider a random vector $x \mapsto (\frac{x}{|x|}, |x|)$ under probability measure 
$\nu/\nu(\rd)$ and let $\nu_{\xi}$ be the conditional distribution of $|x|$ given that 
$\frac{x}{|x|}=\xi$.  
Then 
$(\lambda, \nu_{\xi})$ satisfy \eqref{polar}, 
with $\lambda(B) = \nu(\{x\ne0: \frac{x}{|x|} \in B\})$, $B\in \B (S)$.  

Let $S_{\al}(\rd)$  be the class of $\al$-stable distributions on $\rd$, $0<\al<2$. 
The  characteristic function of $\mu \in S_{\al}(\rd)$ is of the form: when $\al \ne 1$,
\begin{equation}\label{S}
\wh \mu(z)= \exp\left[ -\int_S |\la z,\xi\ra|^{\alpha }\left( 1- i \tan \frac{\pi \alpha}{2} 
\mathrm{sgn}\la z,\xi\ra \right) \lambda_1(d\xi) + i\la z, \tau \ra \right],
\end{equation}
and when $\al= 1$,
\begin{equation}\label{S1}
\wh \mu(z)= \exp\left[ -\int_S \left( |\la z,\xi\ra| + 
i \frac{2}{\pi}\la z,\xi\ra \log|\la z,\xi\ra|  \right) \lambda_1(d\xi) + i\la z, \tau \ra \right],
\end{equation}
where $\lambda_1$ is a finite measure on $S$, called the spectral measure of $\mu$, and 
$\ta \in \rd$ is a shift parameter.  
Recall that $\mu \in S_{\al}(\rd)$ is strictly stable when $\wh \mu(bz)=\wh \mu(z)^{b^{\al}}$ 
for all $b>0$ and $z\in \rd$. 
Let $S_{\al}^0(\rd)$ denote the class of strictly $\al$-stable distributions on $\rd$.
Then
$\mu\in S_{\al}^0(\rd)$ if and only if 
$$
\begin{cases}
\mu\in S_{\al}(\rd)\text{ and }\tau=0,&\text{when } \al\ne  1, \medskip\\
\mu\in S_1(\rd)\text{ and }\int_S\xi\lambda_1(d\xi)=0,&\text{when }\al=1. \medskip\\
\end{cases}
$$
(See, e.g., Theorem 14.10 in \citet{Sato's_book1999}.)

\vskip 5mm
\section{The results}

Recall that $\{X^{(\mu)}_t\}$ denotes a L\'evy process such that $\law(X_1)=\mu \in I(\rd)$.
Let $0<\al <2$.  Consider an improper stochastic integral defined by
$$
\sek t^{-1/\al}d\xtm = \lim_{\varepsilon\downarrow 0, \, T\uparrow\infty} \, 
\int_\varepsilon^T t^{-1/\al}d\xtm
$$
provided the limit in probability exists. In this case we will say that the improper stochastic 
integral is definable, see \citet{Sato2007}. 
Consider a mapping between infinitely divisible 
distributions $\mu \mapsto \Xi _{\al} (\mu)$ given by
\begin{equation}\label{int}
\Xi _{\al} (\mu) = \law \left ( \sek t^{-1/\al}d\xtm\right ).
\end{equation}
Let $\mathfrak D(\Xi _{\al})$ denote the domain of this mapping.  
The following characterization of  $\mathfrak D(\Xi _{\al})$ follows from 
Proposition 5.3 and  Example 4.5 of \citet{Sato2007}.

\begin{thm}\label{D}
(i)
If $0<\al <1$, then
$$
\mathfrak D(\Xi_{\al}) = \left \{ \mu=\mu_{(0,\nu,0)_0} \in I(\rd)
\colon  \int_{\rd}|x|^{\al}\nu(dx)<\infty \right \}.
$$

\n
(ii) 
If $\al=1$, then
\begin{align*}
\mathfrak D(\Xi_{1}) & = \biggl \{ \mu=\mu_{(0,\nu,0)_0} =\mu_{(0,\nu,0)_1}  \in I(\rd)
\colon  \int_{\rd}|x|~\nu(dx)<\infty,   \int_{\rd} x ~\nu(dx)=0, \\
& \hskip 8mm \lim_{\varepsilon\downarrow 0} \int_{|x|\le 1}x \log(|x| \vee \ep)~\nu(dx)    
\ \text{and} \ \lim_{T\to\infty} \int_{|x|> 1} x \log(|x| \wedge T) ~\nu(dx)  \ \text
{exist} \biggr \}\nonumber.\\
\end{align*}

\n
(iii)
If $1<\al<2$, then
$$
\mathfrak D(\Xi_{\al}) = \left \{ \mu=\mu_{(0,\nu,0)_1} \in I(\rd)
\colon \int_{\rd}|x|^{\al}\nu(dx)<\infty \right \}.
$$
\end{thm}

\bigskip

\begin{rem} \label{rem D}
Condition specifying when $\mu \in \mathfrak D(\Xi_{1})$ looks complicated, but there is a simple sufficient condition. Namely, if $\int_{\rd} |x| \, |\log|x|| ~\nu(dx) <\infty$,  $\int_{\rd} x ~\nu(dx)=0$, and $\gamma= \int_{|x|\le 1} x ~\nu(dx)$, then 
$\mu=\mu_{(0,\nu,\gamma)} \in \mathfrak D(\Xi_{1})$.  \\
\end{rem}
\bigskip	

The next three theorems are main results of this paper.  
The first one shows that every strictly stable law can be represented as the law of a stochastic 
integral in \eqref{int}, except for the case $\alpha=1$, when the representation is up to a well specified shift parameter.
The second one connects the integral and series representations of stable distributions. 
The third one gives a subdomain of $\mathfrak D(\Xi_{\al})$ on which $\Xi_{\al}$ is one-to-one. 

\begin{thm}\label{main}
Let $0<\al <2$.

(i) When $\alpha\neq 1$, we have
$$
\Xi_{\al}(\mathfrak D (\Xi_{\al}))=S_{\al}^0(\rd).
$$

(ii) When $\alpha= 1$, we have
\begin{equation}\label{range alpha=1}
\Xi_1(\mathfrak D (\Xi_1)) = \left\{\mu \in S_1^0(\rd)\colon  \tau \in \mathrm{span \ supp}
(\lambda_1)\right\},
\end{equation}
where, respectively, $\lambda_1$ and $\tau$ are the spectral measure and  the shift  
of $\mu$ given by \eqref{S1}.  Here $\mathrm{supp}(\lambda_1)$ denotes the support of $\lambda_1$.  
If $\lambda_1=0$, then we put  $\mathrm{span \ supp}(\lambda_1)=\{0\}$
by convention.
\end{thm}
\bigskip

Let $\mathrm{CP}_0(\rd)$ denote the class of compound Poisson distributions on $\rd$; 
$\mathrm{CP}_0(\rd)= \{\mu_{(0,\nu,0)_0}: \ \nu(\rd)< \infty \}$.

\begin{thm}\label{series}
Let $\mu=\mu_{(0,\nu,0)_0}$ be a compound Poisson distribution on $\rd$ and let $\{X^{(\mu)}_t\}$ 
be the associated L\'evy process.  
Hence $X^{(\mu)}_t = \sum_{j\colon \tau_j\le t} V_j$, where $\{\tau_j\}$ is a sequence of 
arrival times in a Poisson process with rate $\theta=\nu(\rd)$ and $\{V_j\}$ is 
an i.i.d.\ sequence in $\rd$ with the common distribution $\theta^{-1}\nu$, 
independent of $\{\tau_j\}$.
Then
\begin{equation}\label{ser1}
\sek t^{-1/\al}d\xtm = \sum_{j=1}^{\infty}\tau_j^{-1/\al}V_j 
\end{equation}
in the sense that the integral is definable if and only if the series converges a.s.\ 
and then the equality holds a.s.  
Consequently, the series in \eqref{ser1} converges a.s.\ if and only 
if $\mu_{(0, \,\mathcal{L}(V_1), \, 0)_0} \in \mathfrak D (\Xi_{\al})$.  
Furthermore,  $\Xi_{\al}(\mathfrak D (\Xi_{\al}) \cap \mathrm{CP}_0(\rd))$ is a subclass 
of $S^0_{\al}(\rd)$ consisting of distributions representable by either side of \eqref{ser1}.
We have, when $0<\al < 1$, 
\begin{equation}\label{fin0}
\Xi_{\al}(\mathfrak D (\Xi_{\al}) \cap \mathrm{CP}_0(\rd)) = S^0_{\al}(\rd);
\end{equation}
when $\al=1$, 
\begin{equation}\label{fin1}
\Xi_{1}(\mathfrak D (\Xi_{1}) \cap \mathrm{CP}_0(\rd)) = \left\{\mu \in S_1^0(\rd)\colon  
\tau \in \mathrm{span \ supp}(\lambda_1)\right\};
\end{equation}
and when $1<\al<2$, 
\begin{align}\label{fin2}
&  \Xi_{\al}  (\mathfrak D (\Xi_{\al}) \cap \mathrm{CP}_0(\rd)) \\
& = \Big\{\mu \in S^0_{\al}(\rd)\colon \exists q(\xi)>0,  \int_S q(\xi) \xi ~ \lambda_1(d\xi) = 0 \ 
\text{and}   \int_S q(\xi)^{\frac{\al}{\al -1}}~ \lambda_1(d\xi) < \infty  \Big\}  \nonumber \\
&   \subsetneqq S^0_{\al}(\rd). \nonumber
\end{align}
In the above, respectively, $\lambda_1$ and $\tau$ are the spectral measure and  the shift  of $\mu$ 
given by \eqref{S}--\eqref{S1}. 
\end{thm}

\bigskip

The following example sheds some light on the nature of \eqref{fin2}.

\begin{ex}
Let $\mu \in S_{\al}^0(\R^2)$, $1<\al<2$, have the spectral measure  $\lambda_1$ supported by three 
vertices $\xi_1, \xi_2, \xi_3 \in S$ of a proper triangle $\Delta$. Then 
\begin{align*} 
\mu \ \text{is representable by either side of \eqref{ser1}} & \Leftrightarrow  \mu \in \Xi_{\al}  (\mathfrak D (\Xi_{\al}) \cap \mathrm{CP}_0(\R^2)) \\\
& \Leftrightarrow \Delta\text{ is an acute triangle.}   
\end{align*}
Indeed, $\Delta$ is an acute triangle if and only if $0$ belongs to the interior of $\Delta$, that is, 
$p_1\xi_1+ p_2\xi_2+p_3\xi_3=0$ for some $p_1,p_2,p_3>0$.  
Define a function $q$ on $S$ by $q(\xi_i)=p_i/\lambda_1(\{\xi_i\})$, $i=1,2,3$, and 
let $q(\xi)=1$ otherwise. Such function satisfies \eqref{fin2}. 
The converse is clear.
\end{ex}

\bigskip

Let $\mathrm{CP}_0(S)$ stand for the totality of compound Poisson distributions on $\rd$ with 
finite L\'evy measure supported on $S$
and let $\mathrm{CP}_1(S)$ be the totality of infinitely divisible distributions on $\rd$ with
Gaussian covariance matrix $0$, finite L\'evy measure supported on $S$ and mean $0$.
Then $\mathrm{CP}_0(S)\subset \mathfrak D (\Xi_{\al})$ for $0<\alpha<1$,
$\mathrm{CP}_0(S)\cap\mathrm{CP}_1(S)\subset \mathfrak D (\Xi_1)$,
and $\mathrm{CP}_1(S)\subset \mathfrak D (\Xi_{\al})$ for $1<\alpha<2$.

\begin{rem}\label{injective}
The mappings $\Xi_\alpha,0<\alpha<2,$ are not injective.
Let us prove it in the case $\alpha=1$; the proof for $\alpha\neq1$ is similar.
Let $\wt\mu=\wt\mu_{(0,\wt\nu,0)}\in \Xi_1(\mathfrak D(\Xi_1))$.
Then
\begin{align*}
\wt\nu(B) 
&=\int_S \wt\ld (d\xi) \sek \1_B(r\xi)r^{-2}dr\\
&=\sek du\int_S \wt\ld (d\xi) \sek \1_B(u^{-1}r\xi)\delta_1(dr)\\
&=\sek du\int_S \wt\ld (d\xi) \sek \1_B(u^{-1}r\xi)2^{-1}\delta_2(dr).
\end{align*}
Let $\mu_1$ and $\mu_2$ have the L\'evy-Khintchine triplets $(0,\nu_1,0)$ and $(0,\nu_2,0)$, 
where $\nu_1$ and $\nu_2$ have polar decompositions $(\wt\ld,\delta_1)$
and $(\wt\ld,2^{-1}\delta_2)$, respectively.
Then $\mu_1,\mu_2\in\mathfrak D(\Xi_1)$, $\mu_1\neq \mu_2$ and $\Xi_1(\mu_1)=\Xi_1(\mu_2)=\wt\mu$.
See also Remark 6.4 of \citet{Barndorff-NielsenRosinskiThorbjornsen2008}.
However, as shown in what follows, the restrictions $\Xi_\alpha|_{\mathrm{CP}_0(S)}$ with $0<\alpha<1$,
$\Xi_1|_{\mathrm{CP}_0(S)\cap\mathrm{CP}_1(S)}$ and $\Xi_\alpha|_{\mathrm{CP}_1(S)}$ with $1<\alpha<2$
are injective.
\end{rem}

\begin{thm}\label{Poisson}
(i) When $0<\alpha<1$, we have
$$
\Xi_{\al}(\mathrm{CP}_0(S))=S_{\al}^0(\rd)
$$
and the restriction $\Xi_\alpha|_{\mathrm{CP}_0(S)}$ is injective.
Especially, in the case $d=1$, 
$$
S_{\al}^0(\R)=\left\{\law\left(\int_0^\infty t^{-1/\alpha}d(N_1(at)-N_2(bt))\right)
\colon a,b\geq 0\right\},
$$
where $\{N_1(t)\}$ and $\{N_2(t)\}$ are independent Poisson processes with unit rate.

(ii) When $\alpha= 1$, we have
$$
\Xi_1(\mathrm{CP}_0(S)\cap\mathrm{CP}_1(S))=\{\mu\in S_1^0(\rd)\colon 
\text{the shift parameter \(\tau\) in \eqref{S1} is 
\(0\)}\}
$$
and the restriction $\Xi_1|_{\mathrm{CP}_0(S)\cap\mathrm{CP}_1(S)}$ is injective.
Especially, in the case $d=1$, 
$$
\{\mu\in S_1^0(\R)\colon \tau=0\}=\left\{\law\left(\int_0^\infty 
t^{-1}d(N_1(at)-N_2(at))\right)\colon a\geq 0\right\}.
$$

(iii) When $1<\alpha<2$, we have
$$
\Xi_{\al}(\mathrm{CP}_1(S))=S_{\al}^0(\rd)
$$
and the restriction $\Xi_\alpha|_{\mathrm{CP}_1(S)}$ the injective.
Especially, in the case $d=1$,
$$
S_{\al}^0(\R)=\left\{\law\left(\int_0^\infty t^{-1/\alpha}d(N_1(at)-N_2(bt)-(a-b)t)\right)
\colon a,b\geq 0\right\}.
$$
\end{thm}

In the case (ii) of Theorem \ref{Poisson}, the range of 
$\Xi_1|_{\mathrm{CP}_0(S)\cap\mathrm{CP}_1(S)}$ is smaller than that of $\Xi_1$.
Hence there is a subdomain $\mathfrak D$ such that 
$\mathrm{CP}_0(S)\cap\mathrm{CP}_1(S)\subsetneqq \mathfrak D\subsetneqq \mathfrak D(\Xi_1)$,
$\Xi_1(\mathfrak D)=\Xi_1(\mathfrak D(\Xi_1))$ and the restriction $\Xi_1|_{\mathfrak D}$ is injective.
It is an interesting question to characterize the class $\mathfrak D$.

\vskip 5mm
\section{Proofs}

\begin{proof}[Proof of Theorem \ref{D}]
Statements (i) and (iii) follow from  Proposition 5.3 and Example 4.5 of 
\citet{Sato2007}, as does  (ii),
because
$$
\int_\varepsilon^1t^{-1}dt \int_{|x|<t}x\nu(dx) = \int_{|x|\le 1}x \log(|x| \vee \ep)~\nu(dx). 
$$
and
$$
\int_1^Tt^{-1}dt \int_{|x|>t}x\nu(dx) = \int_{|x|> 1} x \log(|x| \wedge T) ~\nu(dx) 
$$
The condition $\int |x| ~ \nu(dx) <\infty$ justifies the interchange of the order of 
integration in these integrals.
\end{proof}
\bigskip

The following lemma is needed for the proof of the next theorem.
\begin{lem}\label{shift}
Let $\lambda$ be a non-zero finite measure on $S$. 
Then, for any $a \in \rd$, $a \in \mathrm{span \ supp}(\lambda)$ if and only if 
$a = \int_{S} \xi f(\xi) ~\lambda(d\xi)$ for some $f \in L^{\infty}(S, \lambda)$.
\end{lem}
\begin{proof}
Consider a linear subspace $H$ of $\rd$ given by
$$
H=\left\{\int_{S} \xi f(\xi) ~\lambda(d\xi)\colon   f \in L^{\infty}(S, \lambda) \right\}
$$
and let $H^{\perp}$ be its  orthogonal complement.
Let $K=\mathrm{span \ supp}(\lambda)$ and  $K^{\perp}$ be the orthogonal complement  of $K$.  
If $p\in H^{\perp}$, then
$$
\int_{S} \la p, \xi \ra f(\xi) ~\lambda(d\xi) = \left\la p, \int_{S} \xi 
f(\xi) ~\lambda(d\xi)\right\ra =0
$$
for all $f \in L^{\infty}(S, \lambda)$, which implies $\la p, \xi \ra =0$ for $\lambda$-almost 
all $\xi$. Hence $p\in K^{\perp}$. 
The converse inclusion, $K^{\perp} \subset H^{\perp}$ is obvious from the displayed equality, 
so that $H^{\perp} = K^{\perp}$. 
This proves $H=K$.
\end{proof}

\bigskip

\begin{proof}[Proof of Theorem \ref{main}]
   
We first show that
$\Xi_{\al}(\mathfrak D (\Xi_{\al}))\subset S_{\al}^0(\rd)$ for all $\alpha\in(0,2)$.
It is enough to prove the strict stability of $Y := \sek t^{-1/\al}dX_t^{(\mu)}$ 
when $\mu \in \mathfrak D (\Xi_{\al})$.
Let $\overline Y$ be an independent copy of $Y$.
Then $\overline Y\eqd\sek t^{-1/\al}d\overline X_t^{(\mu)}$
with an independent copy $\{\overline X_t^{(\mu)}\}$ of $\{X_t^{(\mu)}\}$.
For any $c_1, c_2>0$, we have
\begin{align*}
c_1Y & + c_2\overline Y\\
& = \sek (c_1^{-\al}t)^{-1/\al}dX_t^{(\mu)} + \sek (c_2^{-\al}t)^{-1/\al}d\overline X_t^{(\mu)}\\
& = \sek s^{-1/\al}dX_{c_1^{\al}s}^{(\mu)} + \sek s^{-1/\al}d\overline X_{c_2^{\al}s}^{(\mu)}
 \eqd \sek s^{-1/\al}d\left ( X_s^{(\mu^{c_1^{\al}})} + \overline X_s^{(\mu^{c_2^{\al}})}\right )\\
& \eqd \sek s^{-1/\al}d X_s^{(\mu^{c_1^{\al}}*\mu^{c_2^{\al}})} 
\eqd \sek s^{-1/\al}d X_s^{(\mu^{c_1^{\al}+c_2^{\al}})}\\
& \eqd \sek s^{-1/\al}d X_{(c_1^{\al}+c_2^{\al})s}^{(\mu)}
 = \sek \left ( (c_1^{\al}+ c_2^{\al})^{-1}t\right )^{-1/\al}dX_t^{(\mu)}\\
& = (c_1^{\al}+ c_2^{\al})^{1/\al}\sek t^{-1/\al}dX_t^{(\mu)}
 = (c_1^{\al}+ c_2^{\al})^{1/\al} Y.
\end{align*}
This shows the strict stability of $Y$.

We will also need the following relations. If $\mu=\mu_{(0, \nu, \gamma)} \in \mathfrak D (\Xi_{\al})$,
then $\Xi_\alpha(\mu)=\wt\mu=\wt\mu_{(0, \wt\nu, \wt\gamma)}$, where
\begin{align}\label{}
\wt\nu(B) &=\int_0^{\infty} \int_{\rd} \1_B(t^{-1/\al}x) ~ \nu(dx) dt \label{tn} \\
&=  \int_{\R^d} \int_0^{\infty}\1_B\left(r \frac{x}{|x|}\right) \al r^{-\al-1} |x|^{\al} ~ dr\nu(dx)  
\nonumber \\
&= \int_{S} \lambda(d\xi) \int_0^{\infty} \1_B(r\xi) r^{-\al-1} ~ dr, \nonumber
\end{align}
with
\begin{equation}\label{la}
\lambda(B) = \al  \int_{\R^d} \1_B\left(\frac{x}{|x|}\right) |x|^{\al} ~ \nu(dx),
\end{equation}
and 
\begin{equation}\label{tg}
\wt\gamma = \lim_{\varepsilon\downarrow 0, \, T \uparrow \infty} \, \int_\varepsilon^Tt^{-1/\al}dt
\left(\gamma+\int_\rd x\left(\1_{\{|t^{-1/\al}x|\leq 1\}}-\1_{\{|x|\leq 1\}}\right)
\nu(dx)\right).\end{equation}
It follows that the spectral measure $\lambda_1$ of $\wt\mu$ is given by $\lambda_1=|\Gamma(-\al) 
\cos \frac{\pi \al}{2}| \cdot \lambda$ when $\al \ne 1$ and $\lambda_1=\frac{\pi}{2} \lambda$ 
when $\al=1$, where $\lambda$ is given by \eqref{la}; see the proof of  
Theorem 14.10 in \citet{Sato's_book1999}. 
\medskip

(i) ($\al \ne 1$). We only need to show that 
$S_{\al}^0(\rd) \subset  \Xi_{\al}(\mathfrak D (\Xi_{\al}))$. 

Case $0<\al<1$: \, 
If $\wt\mu=\wt\mu_{(0, \wt \nu, \wt\gm^0)_0} \in S_{\al}^0(\rd)$ then $\wt\gm^0=0$.
Take $\mu=\mu_{(0,\nu,0)_0}$ with $\nu=\alpha^{-1}|\Gamma(-\al) 
\cos \frac{\pi \al}{2}|^{-1}\lambda_1$, where $\lambda_1$ is the spectral measure of $\wt\mu$. 
Then $\mu \in \mathfrak D(\Xi_\alpha)$ and, using \eqref{tn}--\eqref{tg}, 
it is easy to check  that $\Xi_\alpha(\mu)=\wt\mu$.

Case $1<\al<2$: \, 
If $\wt\mu=\wt\mu_{(0, \wt \nu, \wt\gm^1)_1} \in S_{\al}^0(\rd)$ then $\wt\gm^1=0$.
Take $\mu=\mu_{(0,\nu,0)_1}$ with $\nu=\alpha^{-1}|\Gamma(-\al) \cos \frac{\pi \al}{2}|^{-1}
\lambda_1$, where $\lambda_1$ is the spectral measure of $\wt\mu$. 
Then $\mu \in \mathfrak D(\Xi_\alpha)$  and, similarly as above, 
we verify that $\Xi_\alpha(\mu)=\wt\mu$. 
\medskip

(ii) ($\al = 1$). This case is more delicate and its proof is more involved. 
Let $\mu_1 \in S^0_{1}(\rd)$ be given by \eqref{S1} with $\int_S \xi ~\lambda_1(d\xi)=0$ 
and $\tau \in \mathrm{span \ supp}(\lambda_1)$.  Put $\lambda= \frac{2}{\pi} \lambda_1$.

By Lemma \ref{shift}  there is an $f \in L^{\infty}(S, \lambda)$ such that 
$\tau = \int_{S} \xi f(\xi) ~\lambda(d\xi)$. Put $g(\xi)= e^{-f(\xi)}$. 
Then $\varepsilon_0 < g(\xi) < T_0$ $\lambda$-a.e.\ $\xi\in S$ for some $0<\varepsilon_0<1<T_0 $ and
$$
\tau=  -\int_{S} \xi \log g(\xi) ~\lambda(d\xi).
$$
Define a measure $\nu$ on $\rd$ by
$$
\nu(B) =\int_S \lambda(d\xi) \int_0^{\infty}\1_B(r\xi) \, \frac{1}{g(\xi)} \, 
\delta_{g(\xi)}(dr)= \int_S \1_B(g(\xi)\xi) \, \frac{1}{g(\xi)} \,\lambda(d\xi). 
$$
Notice that $\nu$ is a finite measure concentrated on the annulus $\{\varepsilon_0<  |x| < T_0\}$
Therefore,  it clearly satisfies the first, third and forth conditions on $\nu$ of 
Theorem \ref{D}(ii).  
The second condition is also immediate as
$$
\int_{\rd} x ~\nu(dx) = \int_S \xi \lambda(d\xi) =0.
$$  
Thus $\mu=\mu_{(0, \nu, 0)_0} = \mu_{(0, \nu, \gamma)} \in \mathfrak D (\Xi_{1})$, 
where $\gamma=\int_{|x|\le 1} x ~ \nu(dx)$. 
Consider $\Xi_1(\mu) = \wt\mu_{(0, \wt\nu, \wt\gamma)}$.
We have
\begin{align*}
\wt\nu(B) &=\int_0^{\infty} \int_{\rd} \1_B(t^{-1}x) ~ \nu(dx) dt \\
&= \int_0^{\infty} \int_{S} \1_B(t^{-1}g(\xi)\xi) \frac{1}{g(\xi)} ~ \lambda(d\xi) dt  \\
&= \int_{S} \lambda(d\xi) \int_0^{\infty} \1_B(r\xi) r^{-2} ~ dr
\end{align*}
and 
\begin{align*}
\wt\gamma &= \lim_{\varepsilon\downarrow 0 , T \uparrow \infty} \int_\varepsilon^Tt^{-1}dt
\left(\gamma+\int_\rd x\left(\1_{\{|t^{-1}x|\leq 1\}}-\1_{\{|x|\leq 1\}}\right)\nu(dx)\right)\\
&= \lim_{\varepsilon\downarrow 0 , T \uparrow \infty} \int_\varepsilon^Tt^{-1}dt
\int_\rd x\1_{\{|t^{-1}x|\leq 1\}}  ~\nu(dx) \\
&= \lim_{\varepsilon\downarrow 0 , T \uparrow \infty} \int_\rd x   ~\nu(dx)  \int_{\varepsilon 
\vee |x|}^T t^{-1} ~dt = - \int_{S} \xi \log g(\xi) ~\lambda(d\xi) = \tau.
\end{align*}
Note that the shift parameter in \eqref{S1} of $\Xi_1(\mu)$ is $\wt\gamma+c\int_S\xi\lambda(d\xi)$, where $c$ is a constant; see the proof of Theorem 14.10 in \citet{Sato's_book1999}.
Since $\int_S\xi\lambda(d\xi)=0$, the shift parameter of $\Xi_1(\mu)$ is $\wt\gamma$.
Thus $\Xi_1(\mu)  \in S_{1}(\rd)$ has the spectral measure $\frac{\pi}{2} \lambda=\lambda_1$ and a shift $\tau$. This proves that $\Xi_1(\mu)=\mu_1$.

 Conversely, let $\wt\mu \in \Xi_{1}(\mathfrak D (\Xi_{1}))$. 
Then for some $\mu= \mu_{(0, \nu, \gamma)} \in \mathfrak D (\Xi_{1})$, \  
$\wt\mu=\Xi_{\al}(\mu) = \wt\mu_{(0, \wt\nu, \wt\gamma)}$. 
$\wt\mu \in S_{1}^0(\rd)$ has the spectral measure $\lambda_1=\frac{\pi}{2} \lambda$,
where 
$$
\lambda(B) = \int_{\rd} \1_B\left(\frac{x}{|x|}\right) |x| ~ \nu(dx),
$$
and a shift
\begin{align}
\tau &:= \lim_{\varepsilon\downarrow 0 , T \uparrow \infty} \int_\varepsilon^Tt^{-1}dt
\left(\gamma+\int_\rd x\left(\1_{\{|t^{-1}x|\leq 1\}}-\1_{\{|x|\leq 1\}}\right)
\nu(dx)\right) \nonumber\\
&= \lim_{\varepsilon\downarrow 0 , T \uparrow \infty} \int_\varepsilon^Tt^{-1}dt
\int_\rd x\1_{\{|t^{-1}x|\leq 1\}}  ~\nu(dx).  \label{tau1}
\end{align}
Consider a polar decomposition \eqref{polar} of a finite measure $\rho$, given by 
$\rho(dx)=|x| \nu(dx)$, into the spherical component $\lambda$, given above, 
and the radial component $\rho_{\xi}$. $\rho_{\xi}$ are probability measures. 
We have for $0< \varepsilon < 1 < T$, 
\begin{align*}
\tau^T_{\varepsilon} & :=   \int_\varepsilon^Tt^{-1} dt \int_\rd x\1_{\{|t^{-1}x|\leq 1\}}  ~\nu(dx) 
=\left(\int_\varepsilon^1 + \int_1^T \right) t^{-1} dt\int_\rd x\1_{\{|t^{-1}x|\leq 1\}}  ~\nu(dx) \\
&= - \int_{|x| \le 1} x \log(\varepsilon\vee |x|) ~\nu(dx) -  \int_{|x| > 1} x 
\log(T \wedge |x|) ~\nu(dx) \\ 
&= - \int_S \xi ~ \lambda(d\xi) \int_{0}^{\infty} \left( \log(\varepsilon\vee r) \1_{\{r\le 1\}}
+ \log(T\wedge r) \1_{\{r>1\}} \right) ~ \rho_{\xi}(dr)  
\end{align*}
Since $\xi \mapsto \int_{0}^{\infty} \left( \log(\varepsilon\vee r) \1_{\{r\le 1\}} 
+ \log(T\wedge r) \1_{\{r>1\}} \right) ~ \rho_{\xi}(dr) $ 
is a bounded function ($\rho_{\xi}(0,\infty)=1$), 
$\tau^T_{\varepsilon} \in \mathrm{span \ supp}(\lambda)=\mathrm{span \ supp}(\lambda_1)$ 
by Lemma \ref{shift}.  Thus $\tau= \lim_{\varepsilon\downarrow 0 , T 
\uparrow \infty}\tau^T_{\varepsilon} \in \mathrm{span \ supp}(\lambda_1)$.

Finally, if $\Xi_1(\mu)=\delta_{\tau}$ then $\lambda_1=0$, so that $\nu=0$. Hence $\tau=0$ by 
\eqref{tau1}.  The proof of Theorem \ref{main} is complete.
\end{proof}

\bigskip

\begin{proof}[Proof of Theorem \ref{series}]
Let $\Gamma_j=\theta \tau_j$. 
Then $\{\Gamma_j\}$ is a sequence of arrival times in a Poisson process of 
rate one and we can write \eqref{ser1} as
\begin{equation}\label{ser2}
\sek t^{-1/\al}d\xtm = \theta^{1/\al}\sum_{j=1}^{\infty}\Gamma_j^{-1/\al}V_j.
\end{equation}
Applying Theorem 4.1 of \citet{Rosinski2001} we get that the series converges a.s. if and only if 
$E |V|^{\al}= \theta^{-1} \int_{\rd} |x|^{\al} ~ \nu(dx) < \infty$ (where $V=V_1$) and the limit
\begin{equation}\label{cent0}
a := \lim_{T \to \infty} \int_0^T E\left[ t^{-1/\al} V \1_{\{|t^{-1/\al}V|\le 1\} }\right] ~ dt 
\quad \text{exists in } \rd.
\end{equation}
If $\al \ne 1$ then 
\begin{align}\label{cen}
a = \frac{\al}{\al - 1} \lim_{T \to \infty}  E V\1_{\{|V|\le T^{1/\al}\}}(T^{1-1/\al} - |V|^{\al-1}).
\end{align}

Let $0<\al <1$. 
Notice that for every $\varepsilon \in (0,1)$
\begin{align*}
\limsup_{T\to \infty}\,  & T^{1-1/\al}  |E V\1_{\{|V|\le T^{1/\al}\}}|  \le \limsup_{T\to \infty}
\, T^{1-1/\al} E |V|\1_{\{|V|\le \varepsilon T^{1/\al}\}} \\
& + \limsup_{T\to \infty} \, T^{1-1/\al} E |V|\1_{\{ \varepsilon T^{1/\al} < |V|\le T^{1/\al}\}} \\
& \le \limsup_{T\to \infty} \, T^{1-1/\al} (\varepsilon T^{1/\al})^{1-\al} E |V|^{\al} + 
\limsup_{T\to \infty}  E |V|^{\al} 1_{\{ \varepsilon T^{1/\al} < |V|\le T^{1/\al}\}}  \\
& = \varepsilon^{1-\al}E|V|^\alpha.
\end{align*}
Letting $\varepsilon \to 0$ we show that $\limsup_{T\to \infty}\,   
T^{1-1/\al}  E V\1_{\{|V|\le T^{1/\al}\}}=0$.  Therefore, 
$$
a= \frac{\al}{1-\al} \, E \left[ \frac{V}{|V|} |V|^{\al}\right]. 
$$
We conclude that the series in \eqref{ser2} converges a.s. if and only if 
$E |V|^{\al}< \infty$, that is, $\mu \in \mathfrak D (\Xi_{\al})$ when $0<\al <1$. 

If $1<\al <2$ and $\mu \in \mathfrak D (\Xi_{\al})$, 
then $E|V|^{\al}<\infty$. Since also  $E V=0$ we get
\begin{align*}
a &= \frac{\al}{\al - 1} \lim_{T \to \infty}  E V\1_{\{|V|\le T^{1/\al}\}}(T^{1-1/\al} - |V|^{\al-1}) \\
&=\frac{\al}{\al - 1} \lim_{T \to \infty}  E V (T^{1-1/\al} - |V|^{\al-1} \wedge T^{1-1/\al}) =
\frac{\al}{1-\al} \, E \left[ \frac{V}{|V|} |V|^{\al}\right]. 
\end{align*}
Thus, the series  in \eqref{ser2} converges a.s.  
Conversely, if the series  in \eqref{ser2} converges a.s. 
then $E|V|^{\al} < \infty$ and the limit in \eqref{cen} exists. Hence
\begin{align*}
0 &= \lim_{T \to \infty} \left[ E V\1_{\{|V|\le (2T)^{1/\al}\}} ((2T)^{1-1/\al} - |V|^{\al-1}) -
E V\1_{\{|V|\le T^{1/\al}\}} (T^{1-1/\al} - |V|^{\al-1} )  \right] \\
&= ( 2^{1-1/\al} - 1) \lim_{T \to \infty}  T^{1-1/\al} EV\1_{\{|V|\le T^{1/\al}\}} ,
\end{align*}
which implies that $E V=0$.  Thus  $\mu \in \mathfrak D (\Xi_{\al})$.  
We conclude that, when $\al \ne 1$, the series in \eqref{ser2} converges a.s. 
if and only if  $\mu \in \mathfrak D (\Xi_{\al})$. 

Now consider  $\al=1$.  Suppose that the series in \eqref{ser2} converges a.s. 
Then $\int_{\rd} |x| \, \nu(dx)= \theta E |V| < \infty$
and the limit \eqref{cent0} 
\begin{align}\label{cent1}
a := \lim_{T \to \infty} \int_0^T E\left[ t^{-1} V \1_{\{|t^{-1}V|\le 1\} }\right] ~ dt 
= \lim_{T \to \infty}
E V  \1_{\{|V|\le T\}} \log \frac{T}{|V|} 
\end{align}
exists.  
Therefore,
\begin{align*}
0 &= \lim_{T \to \infty} \left[ E V   \1_{\{|V|\le 2T\}} \log \frac{2T}{|V|} - E V \1_{\{|V|\le T\}} 
\log \frac{T}{|V|}\right] \\
&=  \lim_{T \to \infty} \left[ \log 2\, E V  \1_{\{|V|\le 2T\}} +E V \1_{\{T<|V|\le 2T\}} \log 
\frac{T}{|V|}\right] = \log 2 \, E V.
\end{align*}
This shows that  $\int_{\rd} x \, \nu(dx)= \theta E V =0$. 
The condition $\int_{|x| \le 1} |x| \, |\log |x| | ~ \nu(x) < \infty$  
obviously holds because $\nu$ is a finite measure. 
Then for $T>1$ we have
\begin{align*}
&  \int_{|x|> 1} x \log(|x| \wedge T) ~\nu(dx)  = \int_{1<|x|\le T} x \log |x| ~\nu(dx)  +
\log T \int_{|x|> T} x ~\nu(dx) \\
&= \int_{1<|x|\le T} x \log |x| ~\nu(dx)  -  \log T \int_{|x| \le T} x ~\nu(dx) \\
& = -\int_{|x|\le T} x \log \frac{T}{|x|} ~\nu(dx) - \int_{|x|\le 1} x \log |x| ~ \nu(dx) \\
& = - \theta E V  \1_{\{|V|\le T\}} \log \frac{T}{|V|} - \int_{|x|\le 1} x \log |x| ~ \nu(dx).
\end{align*}
Since the limit in \eqref{cent1} exists, 
$\lim_{T \to \infty} \int_{|x|> 1} x \log(|x| \wedge T) ~\nu(dx)$
exists as well.  
Thus $\mu \in  \mathfrak D (\Xi_{1})$. 
Conversely, if $\mu \in  \mathfrak D (\Xi_{1})$, then $E V=0$ and the above computation 
shows that the limit in \eqref{cent1} exists.  
This completes the case $\al=1$.

We have proved that one side in \eqref{ser1} exists if and only if the other one does. 
Now we will show that they are equal. Let $N_t= \max\{j: \tau_j\le t\}$. For any fixed $T>0$ we have
$$
\int_0^T t^{-1/\al} ~dX_t^{(\mu)} = \sum_{j=1}^{N_T} \tau_j^{-1/\al} V_j.
$$
The integral is definable if and only if  the series converges a.s., so  
passing $T\to \infty$ yields the almost sure equality in \eqref{ser1}. 

Relation \eqref{fin0}  is immediate from  the proof of Theorem \ref{main}(i), 
where we constructed for any $\mu_1 \in S^0_{\al}(\rd)$, $0<\al < 1$, 
a distribution $\mu=\mu_{(0,\nu, 0)_0}\in \mathrm{CP}_0(\rd)$  such that  $\Xi_{\al}(\mu) = \mu_1$.  
Similarly, the proof Theorem \ref{main}(ii) gives for any $\mu_1 \in S^0_{1}(\rd)$ 
having $\tau \in \mathrm{span \ supp}(\lambda_1)$, a distribution $\mu=\mu_{(0,\nu, 0)_0}\in 
\mathrm{CP}_0(\rd)$  such that  $\Xi_{1}(\mu) = \mu_1$.  
This shows \eqref{fin0}. 

Now we will prove \eqref{fin2},  $1<\al<2$.  Let $\wt\mu = \Xi_{\al}(\mu)$, where  
$\mu \in \mathfrak D (\Xi_{\al}) \cap \mathrm{CP}_0(\rd)$. 
Hence $\int_{\rd} |x|^{\al} \, \nu(dx)<\infty$
and $\mu=\mu_{(0,\nu,0)_{1}}=\mu_{(0,\nu,0)_{0}}$, so that $\int_{\rd} x ~\nu(dx)=0$ by 
Corollary \ref{0=1}. Moreover, $\nu(\rd)<\infty$. 
Consider a polar decomposition \eqref{polar} of $\rho(dx):=\al |x|^{\al}~\nu(dx)$, 
where the radial components $\rho_{\xi}$ are probability measures and the spherical 
component  $\lambda$ is given by \eqref{la}. Define
$$
q(\xi) = \int_0^{\infty} r^{1-\al} ~\rho_{\xi}(dr).
$$
Then
\begin{align*}
\int_S q(\xi) \xi ~ \lambda(d\xi) &= \int_{S} \int_0^{\infty} \xi r^{1-\al}  \rho_{\xi}(dr) 
\lambda(d\xi) \\ 
&= \int_{\rd} x|x|^{-\al}\rho(dx) = \al \int_{\rd} x\nu(dx)=0,
\end{align*}
and by Jensen's inequality,
\begin{align*}
\int_S q(\xi)^{\frac{\al}{\al-1}} ~ \lambda(d\xi) &= \int_{S} \left( \int_0^{\infty}  
r^{1-\al} ~\rho_{\xi}(dr) \right)^{\frac{\al}{\al-1}}  \lambda(d\xi) \\ 
&\le \int_{S}  \int_0^{\infty}  r^{-\al} ~\rho_{\xi}(dr)  \lambda(d\xi) 
= \al \int_{\rd}  ~\nu(dx)< \infty.
\end{align*}
As noted in the comment following \eqref{la}, the spectral measure $\lambda_1$ of $\wt\mu$ is 
proportional to $\lambda$. Therefore, $q$ satisfies the conditions of \eqref{fin2}.

To prove the converse inclusion, let  $\mu_1$ be a strictly $\al$-stable distribution  with the 
spectral measure $\lambda_1$ such that for some function $q$ the conditions of \eqref{fin2} hold. 
Define a measure $\nu$ by
$$
\nu(B)= c\int_S \lambda_1(d\xi) \int_0^{\infty} \1_{B}(r\xi) ~ q(\xi)^{\frac{\al}{\al-1}} ~ 
\delta_{q(\xi)^{\frac{1}{1-\al}}}(dr),
$$ 
where $c$ is a constant to be determined later. 
Then $\nu(\rd) =  c\int_S q(\xi)^{\frac{\al}{\al-1}}  \lambda_1(d\xi) < \infty$ and 
$$
\int_{\rd} |x|^{\al} ~\nu(dx) = c\int_S \lambda_1(d\xi) < \infty.
$$
Moreover,
$$
\int_{\rd} x ~\nu(dx) = c\int_S q(\xi) \xi ~\lambda_1(d\xi) =0.
$$
Thus $\mu=\mu_{(0,\nu,0)_1} \in \mathfrak D (\Xi_{\al}) \cap \mathrm{CP}_0(\rd)$. Finally, for any $B 
\in  \mathcal{B}(S)$,
\begin{align*}
\int_{\rd} \1_B\left( \frac{x}{|x|}\right) |x|^{\al} \nu(dx) = c\int_S \1_{B}(\xi) ~\lambda_1(d\xi) = 
c \lambda_1(B).
\end{align*}
Hence, taking $c= |\al\Gamma(-\al) \cos \frac{\pi \al}{2}|^{-1}$, we get  $\mu_1 = \Xi_{\al}(\mu)$. 
 
The strict inclusion in \eqref{fin2} is obvious. 
Indeed, let  $\mu \in S^0_{\al}(\rd)$ have the spectral measure $\lambda_1=\delta_v$, 
where $v \in S$.  
Since $\int_S q(\xi) \xi ~\lambda_1(d\xi)= q(v)v\ne 0$ for any positive function $q$, 
$\mu \notin \Xi_{\al}(\mathfrak D (\Xi_{\al}) \cap \mathrm{CP}_0(\rd))$. 
The proof of Theorem \ref{series} is complete.
\end{proof}

\bigskip

\begin{proof}[Proof of Theorem \ref{Poisson}]
(i) The inclusion $\Xi_{\al}(\mathrm{CP}_0(S))\subset S_{\al}^0(\rd)$ is obvious by Theorem \ref{main}.
Let $\wt\mu\in S_{\al}^0(\rd)$.
In the proof of Theorem \ref{main} we constructed a distribution 
$\mu=\mu_{(0,\nu, 0)_0}\in \mathrm{CP}_0(S)$  such that  $\Xi_{\al}(\mu) = \wt\mu$.
Thus $\Xi_{\al}(\mathrm{CP}_0(S))\supset S_{\al}^0(\rd)$.

Let $\mu_1={\mu_1}_{(0,\nu_1, 0)_0},\mu_2={\mu_2}_{(0,\nu_2, 0)_0}\in\mathrm{CP}_0(S)$ 
and $\Xi_{\al}(\mu_1)=\Xi_{\al}(\mu_2)$.
Then by \eqref{tn} and \eqref{la}, the spectral measure of $\Xi_{\al}(\mu_1)$
is $\alpha|\Gamma(-\al) \cos \frac{\pi \al}{2}|\nu_1$.
By the uniqueness of spectral measure, we have $\alpha|\Gamma(-\al) 
\cos \frac{\pi \al}{2}|\nu_1=\alpha|\Gamma(-\al) \cos \frac{\pi \al}{2}|\nu_2$.
Thus $\mu_1=\mu_2$.

In the case $d=1$, $S=\{-1,1\}$ and hence
$$\Xi_{\al}(\mathrm{CP}_0(S))=\left\{\law\left
(\int_0^\infty t^{-1/\alpha}d(N_1(at)-N_2(bt))\right)\colon a,b\geq 0\right\}.$$

(ii) Let $\wt\mu=\wt\mu_{(0,\wt\nu,\wt\gamma)}\in \Xi_1(\mathrm{CP}_0(S)\cap\mathrm{CP}_1(S))$.
Then $\wt\mu=\Xi_1(\mu)\in S_1^0(\rd)$ for some 
$\mu=\mu_{(0,\nu,0)_0}\in\mathrm{CP}_0(S)\cap\mathrm{CP}_1(S)$.
Then $\int_S \xi\nu(d\xi)=0$. Therefore
$$
\wt\gamma=\lim_{\varepsilon\downarrow 0,T\uparrow \infty}\int_\varepsilon^Tt^{-1}dt
\int_\rd x\1_{\{|t^{-1}x|\leq 1\}}\nu(dx)
=\lim_{\varepsilon\downarrow 0,T\uparrow \infty}\int_\varepsilon^Tt^{-1}\1_{\{t\geq  1\}}dt
\int_S \xi\nu(d\xi)=0.
$$
Hence the shift parameter $\tau$ in \eqref{S1} of $\wt\mu$ is $0$.

Conversely, let $\wt\mu\in S_1^0(\rd)$ have the shift parameter $0$.
Take $\mu=\mu_{(0,\nu,0)_0}$ with $\nu=\frac 2\pi\lambda_1$, where $\lambda_1$ 
is the spectral measure of $\wt\mu$. 
Then $\mu \in \mathrm{CP}_0(S)\cap\mathrm{CP}_1(S)$ and $\Xi_1(\mu)=\wt\mu$.

The rest of the proof of this case is similar to the case (i).

(iii) This case is similar to the case (i).
\end{proof}
\vskip 5mm
\section{Final remarks}

The following proposition and remark are concerned with the limits of 
ranges of iterations of the mappings $\Xi_\alpha$, $0<\alpha<2$.  
The composition $\Xi_\alpha^2= \Xi_\alpha \circ \Xi_\alpha$\,  is defined on the domain 
$\mathfrak D(\Xi_\alpha^2)=\{\mu\in\mathfrak D(\Xi_\alpha)\colon 
\Xi_\alpha(\mu)\in\mathfrak D(\Xi_\alpha)\}$.

\begin{prop}
Let $0<\alpha<2$.
Then $\mathfrak D(\Xi_\alpha^2)=\{\delta_0\}$, so that
$$
\Xi_\alpha^m(\mathfrak D(\Xi_\alpha^m))=\{\delta_0\} \quad \text{for every } m\ge 2.
$$
\end{prop}
\begin{proof}
Let   $\mu \in \mathfrak D(\Xi_\alpha^2)$ have the L\'evy measure $\nu$.  
Since $\Xi_\alpha(\mu)\in\mathfrak D(\Xi_\alpha)$, its L\'evy measure $\wt\nu$ has  
finite $\al$-th moment by Theorem~\ref{D}. 
From \eqref{tn} this is only possible when $\nu=0$. 
Hence $\mu$ is a point mass distribution belonging to $\mathfrak D(\Xi_\alpha)$. 
By Theorem~\ref{D} (or simply by \eqref{int}), 
$\mu=\delta_0$.
\end{proof}

\medskip

\begin{rem}\label{iteration}
Given a stochastic integral mapping
$\Phi_f$ in \eqref{si},
an interesting problem is to determine the limit $\bigcap_{m=1}^\infty \Phi_f^m(\mathfrak D(\Phi_f^m))$.
Recently, for many integrands $f$, the descriptions of $\bigcap_{m=1}^\infty 
\Phi_f^m(\mathfrak D(\Phi_f^m))$ have been obtained.
In many cases, $\bigcap_{m=1}^\infty \Phi_f^m(\mathfrak D(\Phi_f^m))$ is the class 
$L_\infty(\rd)$ of completely selfdecomposable distributions or its subclasses;
see \citet
{MaejimaSato2009,MaejimaUeda2009e,MaejimaUeda2009c,AoyamaLindnerMaejima2009,Sato2011,Sato2012}.
Also, in some cases, $\bigcap_{m=1}^\infty \Phi_f^m(\mathfrak D(\Phi_f^m))$ is the class 
$L_\infty(b,\rd)$ of completely semi-selfdecomposable distributions with span $b>1$,
which is the closure of the class of semi-stable distributions with span $b$ under convolution and 
weak convergence;
see \citet{MaejimaUeda2009}.
Other examples are found in Section 4 of \citet{Sato2011}.
However, except of $\Xi_\alpha$, we do not know any example of a stochastic integral 
mapping $\Phi_f$ satisfying
$\Phi_f(\mathfrak D(\Phi_f))\neq\{\delta_0\}$ and $\bigcap_{m=1}^\infty 
\Phi_f^m(\mathfrak D(\Phi_f^m))=\{\delta_0\}$.
\end{rem}

\vskip 5mm


\begin{thebibliography}{18}
\expandafter\ifx\csname natexlab\endcsname\relax\def\natexlab#1{#1}\fi
\expandafter\ifx\csname url\endcsname\relax
  \def\url#1{{\tt #1}}\fi

\bibitem[Aoyama et~al.(2010)Aoyama, Lindner, and
  Maejima]{AoyamaLindnerMaejima2009}
T.~Aoyama, A.~Lindner, and M.~Maejima.
\newblock A new family of mappings of infinitely divisible distributions
  related to the {G}oldie-{S}teutel-{B}ondesson class.
\newblock {\em Electron. J. Probab.} {\bfseries  15} \penalty0 (2010),
  \penalty0 1119--1142.

\bibitem[Aoyama and Maejima(2007)]{AoyamaMaejima2007}
T.~Aoyama and M.~Maejima.
\newblock {Characterizations of subclasses of type $G$ distributions on
  $\mathbb{R}^d$ by stochastic integral representations.}
\newblock {\em Bernoulli} {\bfseries  13} \penalty0 (2007), \penalty0 148--160.

\bibitem[Aoyama et~al.(2008)Aoyama, Maejima and
  Rosi\'nski]{AoyamaMaejimaRosinski2008}
T.~Aoyama, M.~Maejima, and J.~Rosi\'nski.
\newblock {A subclass of type $G$ selfdecomposable distributions on $\mathbb
  R^d$.}
\newblock {\em J. Theor. Probab.} {\bfseries  21} \penalty0 (2008), \penalty0
  14--34.

\bibitem[Aoyama et~al.(2011)]{AoyamaMaejimaUeda2011}
T.~Aoyama, M.~Maejima, and Y.~Ueda.
\newblock {Several forms of stochastic integral representations of 
gamma random variables and related topics.}
\newblock {\em Probab. Math. Statist.} {\bfseries  31} \penalty0 (2011), \penalty0
  99-118.  

\bibitem[Barndorff-Nielsen et~al.(2006)Barndorff-Nielsen, Maejima, and
  Sato]{Barndorff-NielsenMaejimaSato2006}
O.E. Barndorff-Nielsen, M.~Maejima, and K.~Sato.
\newblock {Some classes of multivariate infinitely divisible distributions
  admitting stochastic integral representations.}
\newblock {\em Bernoulli} {\bfseries  12} \penalty0 (2006), \penalty0 1--33.

\bibitem[Barndorff-Nielsen et~al.(2008)Barndorff-Nielsen, Rosinski, and
  Thorbj{\o}rnsen]{Barndorff-NielsenRosinskiThorbjornsen2008}
O.E. Barndorff-Nielsen, J.~Rosinski, and S.~Thorbj{\o}rnsen.
\newblock {General $\Upsilon$-transformations.}
\newblock {\em ALEA, Lat. Am. J. Probab. Math. Stat.} {\bfseries  4} \penalty0
  (2008), \penalty0 131--165.

\bibitem[Ichifuji et~al.(2010)]{IchifujiMaejimaUeda2010}
K.~Ichifuji and M.~Maejima and Y.~Ueda.
\newblock{Fixed points of mappings of infinitely divisible distributions on $\mathbb R^d$.}
\newblock{Statist. Probab. Lett.}
{\bfseries  80} \penalty0
  (2010), \penalty0 1320--1328.
  
\bibitem[Jurek(1985)]{Jurek1985}
Z.J.~Jurek.
\newblock{Relations between the {$s$}-selfdecomposable and selfdecomposable measures.}
\newblock{\em Ann. Probab.} {\bfseries  13}
 \penalty0 (1985), \penalty0 592--608.

\bibitem[Jurek and Vervaat(1983)]{JurekVervaat1983}
Z.J. Jurek and W.~Vervaat.
\newblock {An integral representation for selfdecomposable Banach space valued
  random variables.}
\newblock {\em Z. Wahrscheinlichkeitstheor. Verw. Geb.} {\bfseries  62}
\penalty0 (1983),
  \penalty0 247--262.
 
\bibitem[Maejima and Sato(2009)]{MaejimaSato2009}
M.~Maejima and K.~Sato.
\newblock {The limits of nested subclasses of several classes of infinitely
  divisible distributions are identical with the closure of the class of stable
  distributions.}
\newblock {\em Probab. Theory Relat. Fields} {\bfseries  145} \penalty0 (2009),
  \penalty0 119--142.

\bibitem[Maejima and Ueda(2009)]{MaejimaUeda2009}
M.~Maejima and Y.~Ueda.
\newblock {Stochastic integral characterizations of semi-selfdecomposable
  distributions and related Ornstein-Uhlenbeck type processes}.
\newblock {\em Commun. Stoch. Anal.} {\bfseries  3} \penalty0 (2009), \penalty0
  349--367.

\bibitem[Maejima and Ueda(2010)]{MaejimaUeda2009e}
M.~Maejima and Y.~Ueda.
\newblock {Compositions of mappings of infinitely divisible distributions with
  applications to finding the limits of some nested subclasses}.
\newblock {\em Electron. Commun. Probab.} {\bfseries  15} \penalty0 (2010),
  \penalty0 227--239.

\bibitem[Maejima and Ueda(2011)]{MaejimaUeda2009c}
M.~Maejima and Y.~Ueda.
\newblock {Nested subclasses of the class of $\alpha$-selfdecomposable
  distributions}.
\newblock {\em Tokyo J. Math.} {\bfseries  34} \penalty0 (2011), \penalty0
  383--406.

\bibitem[Rocha-Arteaga and Sato(2001)]{Rocha-ArteagaSato2003}
A.~Rocha-Arteaga and K.~Sato.
\newblock{\em Topics in Infinitely Divisible Distributions and L\'evy Processes.}
\newblock{Aportaciones Matem\'aticas, Investigaci\'on 17, Sociedad Matem\'atica Mexicana},
\penalty0 (2003).

\bibitem[Rosinski(1990)]{Rosinski1990}
J.~Rosinski.
\newblock {On series representations of infinitely divisible random vectors.}
\newblock {\em Ann. Probab.} {\bfseries  18} \penalty0 (1990), \penalty0
  405--430.

\bibitem[Rosi\'nski(2001)]{Rosinski2001}
J.~Rosi\'nski.
\newblock {Series representations of L\'evy processes from the perspective of
  point processes.}
\newblock In: {Barndorff-Nielsen, Ole E. (ed.) et al., L\'evy processes. Theory
  and applications. Boston: Birkh\"auser. 401-415}, 2001.
  
\bibitem[Samorodnitsky and Taqqu(1994)]{S-T_book1994} 
G.~Samorodnitsky and M.S.~Taqqu. 
\newblock {\em {Stable Non-Gaussian Random Processes. Stochastic Modeling.}} New York: Chapman \& Hall. Stochastic models with infinite variance.

\bibitem[Sato(1999)]{Sato's_book1999}
K.~Sato.
\newblock {\em {L\'evy Processes and Infinitely Divisible Distributions.}}
\newblock {Cambridge University Press}, {Cambridge}, 1999.

\bibitem[Sato(2006)]{Sato2006}
K.~Sato.
\newblock {Two families of improper stochastic integrals with respect to L\'evy processes.}
\newblock {\em ALEA, Lat. Am. J. Probab. Math. Stat.} {\bfseries  1} \penalty0
  (2006), \penalty0 47--87.

\bibitem[Sato(2007)]{Sato2007}
K.~Sato.
\newblock {Transformations of infinitely divisible distributions via improper stochastic integrals.}
\newblock {\em ALEA, Lat. Am. J. Probab. Math. Stat.} {\bfseries  3} \penalty0
  (2007), \penalty0 67--110.

\bibitem[Sato(2010)]{Sato2010}
K.~Sato.
\newblock {Fractional integrals and extensions of selfdecomposability.}
\newblock {Lecture Notes in Mathematics 2001, L\'evy matters I. 1-91,
  Springer}, 2010.

\bibitem[Sato(2011)]{Sato2011}
K.~Sato.
\newblock {Description of limits of ranges of iterations of stochastic integral
  mappings of infinitely divisible distributions.}
\newblock {\em ALEA, Lat. Am. J. Probab. Math. Stat.} {\bfseries  8} \penalty0
  (2011), \penalty0 1--17.

\bibitem[Sato(2013)]{Sato2012}
K.~Sato.
\newblock {Inversions of infinitely divisible distributions and conjugates of
  stochastic integral mappings.}
\newblock {\em {\rm To appear in} J. Theoret. Probab.}  \penalty0 (2013).

\bibitem[Sato and Ueda(2013)]{SatoUeda2012}
K.~Sato and Y.~Ueda.
\newblock{Weak drifts of infinitely divisible distributions and their applications.}
\newblock {\em J. Theoret. Probab.} {\bfseries  26} \penalty0 (2013),
  \penalty0 885--898.

\bibitem[Sato and Yamazato(1984)]{SatoYamazato1984}
K.~Sato and M.~Yamazato.
\newblock {Operator-selfdecomposable distributions as limit distributions of
  processes of Ornstein-Uhlenbeck type.}
\newblock {\em Stoch. Proc. Appl.} {\bfseries  17} \penalty0 (1984), \penalty0
  73--100.

\bibitem[Wolfe(1982)]{Wolfe1982}
S.J. Wolfe.
\newblock {On a continuous analogue of the stochastic difference equation
  $X_n=\rho X_{n-1}+B_n$.}
\newblock {\em Stoch. Proc. Appl.} {\bfseries  12} \penalty0 (1982), \penalty0
  301--312.

\end{thebibliography}
\end{document}